\documentclass[12pt]{amsart}
\usepackage{amsmath,amssymb,amscd,graphicx,verbatim,rotating}
\usepackage[colorlinks=true, pdfstartview=FitV, linkcolor=blue, citecolor=blue, urlcolor=blue]{hyperref}
\pdfoutput=1
\usepackage{mathpazo}
\theoremstyle{plain}
\newtheorem*{theorem}{Theorem}

\newcommand{\tre}{\text{Re}}

\newcommand{\lodd}{\lambda_{\text{odd}}}
\newcommand{\todd}{\tau_{\text{odd}}}

\theoremstyle{definition}
\newtheorem{lemma}{Lemma}
\newtheorem*{remark}{Remark}

\begin{document}
\title[Elementary Deuring-Heilbronn]{Elementary Deuring-Heilbronn Phenomenon}
\author{Jeffrey Stopple}
\begin{abstract}
Adapting a technique of Pintz, we give an elementary demonstration of the Deuring phenomenon: a zero of $\zeta(s)$ off the critical line gives a lower bound on $L(1,\chi)$.  The necessary tools are Dirichlet's \lq method of the hyperbola\rq, Euler summation, summation by parts, and the Polya-Vinogradov inequality.
\end{abstract}
\email{stopple@math.ucsb.edu}\address{Mathematics Department, UC Santa Barbara, Santa Barbara CA 93106}
\keywords{Landau-Siegel zero, Deuring-Heilbronn phenomenon}
\subjclass[2000]{11M20, 11M26}

\maketitle

\subsection*{Introduction}

In a long series of papers in \emph{Acta Arithmetica}, J\'anos Pintz gave remarkable elementary proofs of theorems concerning $L(s,\chi)$, $\chi$ the Kronecker symbol attached to a fundamental discriminant $-D$.  These include theorems of Hecke, Landau, Siegel, Page, Deuring, and Heilbronn \cite{Pintz, PintzI, PintzII, PintzIII, PintzIV, PintzV}.  In \cite{PintzIII}, for example, he gives his version of the Deuring phenomenon \cite{Deuring}: Under the very strong assumption that the class number $h(-D)\le \log^{3/4}(D)$, he obtains a zero free region for $\zeta(s)L(s,\chi)$.  As the reviewer in \emph{Math.\ Reviews} noted, by Siegel's Theorem this can hold for only finitely many $D$ (with an ineffective constant.)  Subsequently the Goldfeld-Gross-Zagier Theorem shows this can happen for only finitely many $D$ with an effective constant\footnote{In fact there are 61 such fundamental discriminants, all with $-1555\le -D$.}.  This is unfortunate, as the proof Pintz gave actually depends on the fact that the exponent of the class group $\mathcal C(-D)$ (v.\ the order) is small.  

In \cite{PintzIV} he gives an elementary version  of (the contrapositive of) the Heilbronn phenomenon \cite{Heilbronn}:\ a zero off the critical line of an $L$-function $L(s,\chi_k)$ attached to any primitive real character can be used to give lower bounds on $L(1,\chi)$.  The same \emph{Math.\ Reviews} reviewer called the proof \lq\lq ingenious and quite brief.\rq\rq\footnote{See also \cite{Motohashi}, \cite[\S 4.2]{Motobook} for an elementary proof by Motohashi which is based on the Selberg sieve.}

Pintz's idea is very roughly as follows:  With $\lambda$ denoting the Liouville function, the convolution $1*\lambda$ is the characteristic function of squares.  Thus for $\rho$ a hypothetical zero of $L(s,\chi_k)$ with $\tre(\rho)>1/2$, one can consider finite sums of the form
\[
\sum_{n<X}\frac{\chi_k(n)}{n^\rho}1*\lambda(n).
\]
Since $\chi_k(m^2)=1$ or $0$, one can compare this sum to a partial sum of $\zeta(2\rho)$, and obtain a lower bound.
Pintz decomposes the sum into two pieces, carefully chosen so that $L(\rho,\chi_k)=0$ shows one piece is not too big, and therefore the other piece is not too small.  But if $L(1,\chi)$ were small due to the existence of a Landau-Siegel zero, $\chi$ would be a good approximation to $\lambda$, and (he can show) this second term would necessarily be small.

In this paper we adapt the method of \cite{PintzIV} to apply to $\zeta(s)$, and thus give an elementary demonstration of the Deuring phenomenon.  Because $\zeta(s)$ does not converge even conditionally in the critical strip, we assume first that $D$ is even, and consider instead
\[
\phi(s)=\left(2^{1-s}-1\right)\zeta(s)=\sum_n\frac{(-1)^n}{n^s}.
\]  
Suppose $\rho=\beta+i\gamma$ is a zero of $\zeta(s)$  off the critical line.  Let $\delta/2\pi$ be the fractional part of $\log 2\cdot  \gamma/2\pi$ so that for integer $n$,
\begin{gather*}
\log 2\cdot \gamma =2\pi n+\delta,\\
-\pi<\delta\le \pi,\\
2^{-i\gamma}=\exp(-i\delta).
\end{gather*}

\begin{theorem}  If $\beta>7/8$ and $|\delta|>\pi/100$, then for any real primitive character $\chi$ modulo $D\equiv 0\bmod 4$, $D>10^9$, we have the lower bound
\[
L(1,\chi)>\frac{1}{5400\cdot U^{12(1-\beta)}\log^3U},
\]
where $U=|\rho|D^{1/4}\log D$.  
\end{theorem}

The proof actually gives some kind of nontrivial bound as long as $\beta>5/6$.  We assume $\beta>7/8$ simply to get a precise constant in the theorem.  

In the last section we discuss general $D$, adapting the proof with Ramanujan sums $c_q(n)$ for a fixed prime $q|D$.

\subsection*{Arithmetic Function Preliminaries}

Generalizing Liouville's $\lambda$ function, we begin by defining $\lodd(n)$ via
\[
\lodd(n)
=
\begin{cases}
0&\text{ if } n \text{ is even}\\
\lambda(n)&\text{ if } n \text{ is odd.}
\end{cases}
\]
So
\[
\sum_{n=1}^\infty\frac{\lodd(n)}{n^s}=\frac{\zeta(2s)}{\zeta(s)}\cdot (1+2^{-s}),
\]
and the convolution $1*\lodd(n)$ satisfies
\[
1*\lodd(n)
=
\begin{cases}
1&\text{ if }n=m^2\text{ or }n=2m^2\\
0&\text{ otherwise}.
\end{cases}
\]

With $\tau(n)$ the divisor function and $\nu(n)$ the number of distinct primes dividing $n$, we have that
\[
1*\lambda(n)=\sum_{d|n}2^{\nu(d)}\lambda(d)\tau(n/d).
\]
(One needs to verify this only for $n=p^k$ as both sides are multiplicative.)\ \   We generalize this by defining $\todd(n)$ to be the number of odd divisors of $n$, so that
\[
1*\lodd(n)=\sum_{d|n}2^{\nu(d)}\lodd(d)\todd(n/d).
\]
(For $n$ odd this follows from $\lodd(d)=\lambda(d)$ and $\todd(n/d)=\tau(n/d)$, while for $n=2^k$ both sides are equal $1$.)\ \ 

Following Pintz we define, relative to the quadratic character $\chi$ modulo $D$, sets
\begin{gather*}
A_j=\{u \text{ such that } p|u\Rightarrow \chi(p)=j\}\quad\text{ for }\quad j=-1,0,1\\
C=\{c=ab\,|\, a\in A_1, b\in A_0\}.
\end{gather*}
We are assuming that $2\in A_0$, so integers in $A_{-1}$ and $A_1$ are odd.  We factor an arbitrary $n$ as
\[
n=abm=cm,\quad\text{ where }a\in A_1, b\in A_0, m\in A_{-1}, c\in C.
\]
We then see that for
\begin{alignat*}{2}
&a\in A_1,&\quad1*\chi(a)=&\tau(a)=\todd(a),\\
&b\in A_0,&\quad1*\chi(b)=&1,\\
&m\in A_{-1},&\quad1*\chi(m)=&1*\lambda(m)=1*\lodd(m).
\end{alignat*}
Using this and multiplicativity, for $n=abm=cm$ as above we see that
\begin{multline}\label{Eq1}
1*\lodd(n)=1*\lodd(a)\cdot 1*\lodd(b)\cdot 1*\lodd(m)=\\
\left(\sum_{a^\prime|a}2^{\nu(a^\prime)}\lodd(a^\prime)\cdot 1*\chi(a/a^\prime)\right)\left(\sum_{b^\prime|b}\lodd(b^\prime)\cdot 1*\chi(b/b^\prime)\right)\cdot 1*\chi(m)\\
=\sum_{\substack{c^\prime|c,\\c^\prime=a^\prime b^\prime}}2^{\nu(a^\prime)}\lodd(c^\prime)\cdot 1*\chi(n/c^\prime).
\end{multline}

\subsection*{Lower Bounds}

\begin{lemma}  \label{L:Lemma1}
\[
\frac{1}{25}\cdot \frac{\zeta(4\beta)}{\zeta(2\beta)}-U^{6-12\beta}
\le \left|\sum_{n\le U^{12}}\frac{(-1)^n\cdot 1*\lodd(n)}{n^\rho}\right|
\]
\end{lemma}
\begin{proof}
We have
\begin{multline*}
\left|\sum_{n\le U^{12}}\frac{(-1)^n\cdot 1*\lodd(n)}{n^\rho}\right|
\ge \\
\left|\sum_{n=1}^\infty\frac{(-1)^n\cdot 1*\lodd(n)}{n^\rho}\right|
-\left|\sum_{ U^{12}<n}\frac{(-1)^n\cdot 1*\lodd(n)}{n^\rho}\right|.
\end{multline*}
Now
\[
\sum_{n=1}^\infty\frac{(-1)^n\cdot 1*\lodd(n)}{n^\rho}=
\sum_{m=1}^\infty\frac{(-1)^{m^2}}{m^{2\rho}}+\sum_{m=1}^\infty\frac{(-1)^{2m^2}}{2^\rho m^{2\rho}}.
\]
Observe that $(-1)^{m^2}=(-1)^m$, and of course $(-1)^{2m^2}=1$.
This gives
\[
\left(2^{1-2\rho}-1\right)\zeta(2\rho)+2^{-\rho}\zeta(2\rho)\\
=\left(1+2^{-\rho}\right)\left(2^{1-\rho}-1\right)\zeta(2\rho).
\]
We compare Euler products to see
\[
\frac{1}{\left|\zeta(2\rho)\right|}<\frac{\zeta(2\beta)}{\zeta(4\beta)},\quad\text{or}\quad
|\zeta(2\rho)|>\frac{\zeta(4\beta)}{\zeta(2\beta)}.
\]
Finally a calculation in \emph{Mathematica} shows that
\[
\left|\left(1+2^{-\rho}\right)\left(2^{1-\rho}-1\right)\right|>\frac{1}{25}
\]
as long as $|\delta|>\pi/100$.  This gives the main term of the Lemma.

Meanwhile
\[
\left|\sum_{ U^{12}<n}\frac{(-1)^n\cdot 1*\lodd(n)}{n^\rho}\right|\le
\left|\sum_{ U^{6}<m}\frac{(-1)^{m}}{m^{2\rho}}\right|+\left|\frac{1}{2^\rho}\sum_{ U^{6}/\sqrt{2}<m}\frac{1}{m^{2\rho}}\right|.
\]
The first sum on the right is bounded by $U^{-12\beta}$, by Abel's inequality.  And the second sum, via Euler summation formula \cite[Theorem 3.2 (c)]{Apostol} is $O(U^{6-12\beta})$.  In fact, the proof given there shows the implied constant can be taken as $1/(\sqrt{2}(2\beta-1))<1$ for $\beta>7/8$.
\end{proof}

\subsection*{Upper Bounds}
We now follow Pintz in writing
\begin{multline*}
\left|\sum_{n\le U^{12}}\frac{(-1)^n}{n^\rho}\cdot 1*\lodd(n)\right|\\
=\left|\sum_{n\le U^{12}}\frac{(-1)^n}{n^\rho} \sum_{c\in C,c|n}2^{\nu(a)}\lodd(c)\cdot 1*\chi(n/c)\right|,
\end{multline*}
via (\ref{Eq1}).
We change variables $n=rc$, and use the fact that for odd $c$ we have
\[
(-1)^{rc}=(-1)^r, \text{ and }\lodd(c)=0\text{ unless }c\text{ is odd}.
\]
(The fact that $(-1)^n$ is not a multiplicative function is the reason we've introduced $\lodd(n)$.)\ \ This is equal to
\begin{multline*}
=\left|\sum_{c\le U^{12},c\in C}\frac{2^{\nu(a)}\lodd(c)}{c^\rho}\sum_{r\le U^{12}/c}\frac{(-1)^r}{r^\rho}\cdot 1*\chi(r)\right|\le\Sigma^\prime_1+\Sigma^\prime_2,
\end{multline*}
where
\begin{align*}
\Sigma^\prime_1=&\sum_{\substack{c\le U^{6}\\ c\in C}}\frac{2^{\nu(a)}}{c^\beta}\left|\sum_{r\le U^{12}/c}\frac{(-1)^r}{r^\rho}\cdot 1*\chi(r)\right|\\
\Sigma^\prime_2=&\sum_{\substack{U^{6}< c\le U^{12}\\ c\in C}}\frac{2^{\nu(a)}}{c^\beta}\sum_{r\le U^{12}/c}\frac{1*\chi(r)}{r^\beta}.
\end{align*}
Using the inequalities
\begin{gather*}
2^{\nu(a)}\le 1*\chi(c)\le\todd(c)\le \tau(c),\\
1*\chi(r)\le \tau(r),
\end{gather*}
and dropping the condition $c\in C$ in the outer sums,
we see that
\begin{align*}
\Sigma_1^\prime\le \Sigma_1=&\sum_{n\le U^{6}}\frac{\tau(n)}{n^\beta}\left|\sum_{r\le U^{12}/n}\frac{(-1)^r}{r^\rho}\cdot 1*\chi(r)\right|,\\
\Sigma^\prime_2\le \Sigma_2=&\sum_{U^{6}< n\le U^{12}}\frac{1*\chi(n)}{n^\beta}\sum_{r\le U^{12}/n}\frac{\tau(r)}{r^\beta}.
\end{align*}
\begin{remark}
The main idea of the proof is to use the fact that $\zeta(\rho)=0$ to show that $\Sigma_1$ can not be too big.  This then implies that $\Sigma_2$ can not be too small, from which we can lower bound $L(1,\chi)$.
\end{remark}
\begin{lemma} \label{L:Lemma2} We estimate the inner sum  in $\Sigma_1$ as
\[
\left|\sum_{r\le y}\frac{(-1)^r}{r^\rho}\sum_{d|r}\chi(d)\right|< 
\frac{2}{3}\cdot y^{1/2-\beta}|\rho|D^{1/4}\log D\log (y/\sqrt{D}).
\]
\end{lemma}
\begin{proof}  We write $(-1)^r=(-1)^{ld}$.  Since we're assuming $D$ is even, $\chi(d)=0$ unless $d$ is odd and so $(-1)^{ld}=(-1)^l$.  This gives
\begin{gather*}
\left|\sum_{r\le y}\frac{(-1)^r}{r^\rho}\sum_{d|r}\chi(d)\right|=
\left|\sum_{d\le y}\frac{\chi(d)}{d^\rho}\sum_{l\le y/d}\frac{(-1)^l}{l^\rho}\right|\\
\le\left|\sum_{d\le z}\frac{\chi(d)}{d^\rho}\sum_{l\le y/d}\frac{(-1)^l}{l^\rho}\right|+
\left|\sum_{l\le y/z}\frac{(-1)^l}{l^\rho}\sum_{z<d\le y/l}\frac{\chi(d)}{d^\rho}\right|.
\end{gather*}
The parameter $z$ will be chosen later to make these two terms approximately the same size.
Summation by parts \cite[Theorem 4.2]{Apostol} gives
\[
\phi(s)= \sum_{l=1}^{y/d} \frac{(-1)^l}{l^s}-\frac{S(y/d)}{(y/d)^s}+s\int_{y/d}^\infty\frac{S(x)-S(y/d)}{x^{s+1}}dx,
\]
where $S(x)=\sum_{n\le x} (-1)^n$ is $-1$ or $0$.  Set $s=\rho$ and use $\phi(\rho)=0$; we bound the integral getting
\begin{gather*}
\left|s\int_{y/d}^\infty\frac{S(x)-S(y/d)}{x^{s+1}}dx\right|\le\frac{|\rho|}{\beta (y/d)^\beta}\\
\left|\frac{S(y/d)}{(y/d)^s}\right|\le\frac{1}{(y/d)^\beta}.
\end{gather*}
So we claim
\[
\left| \sum_{l=1}^{y/d} \frac{(-1)^l}{l^\rho}\right|\le\frac{|\rho|}{\beta (y/d)^\beta},
\]
since $1<1/\beta$ and \cite{Gourdan} shows that $10^{12}<|\rho|$.

Thus we can estimate the first term in the previous sum
\[
\left|\sum_{d\le z}\frac{\chi(d)}{d^\rho}\sum_{l\le y/d}\frac{(-1)^l}{l^\rho}\right|\le
\sum_{d\le z}\frac{1}{d^\beta}\cdot \frac{|\rho|}{\beta (y/d)^\beta}\\
= \frac{z|\rho|}{y^\beta\cdot \beta}.
\]
Another summation by parts gives
\[
\sum_{z<d\le y/l}\frac{\chi(d)}{d^s}=\frac{S_D(y/l)}{(y/l)^s}-\frac{S_D(z)}{z^s}+s\int_{z}^{y/l}\frac{S_{D}(x)-S_{D}(\sqrt{y})}{x^{s+1}}dx,
\]
where $S_D(x)=\sum_{n\le x}\chi(n)$.  By the Polya-Vinogradov inequality \cite[Theorem 8.21]{Apostol}, $\left|S_D(x)\right|<\sqrt{D}\log D$.
Neglecting the boundary terms as before, we bound the integral as
\[
\left|\sum_{z<d\le y/l}\frac{\chi(d)}{d^\rho}\right|\le \frac{|\rho|\sqrt {D}\log D}{\beta z^{\beta}},
\]
and so bound the second sum above as
\begin{multline*}
\left|\sum_{l\le y/z}\frac{(-1)^l}{l^\rho}\sum_{z<d\le y/l}\frac{\chi(d)}{d^\rho}\right|\\
\le
\sum_{l\le y/z}\frac{|\rho|\sqrt {D}\log D}{\beta l^\beta z^{\beta}}
=\frac{|\rho|\sqrt {D}\log D}{\beta}\sum_{l\le y/z}\frac{1}{ l^\beta z^{\beta}}.
\end{multline*}
Now
\begin{multline*}
\sum_{l\le y/z}\frac{1}{ l^\beta z^{\beta}}=\frac{y^{1-\beta}}{z}\sum_{l\le y/z}\frac{1}{ l^\beta (y/z)^{1-\beta}}\\
<\frac{y^{1-\beta}}{z}\sum_{l\le y/z}\frac{1}{l^\beta\cdot l^{1-\beta}}\sim\frac{y^{1-\beta}\log(y/z)}{z},
\end{multline*}
where the inequality follows since $l<y/z$.  This gives, for the second sum, the bound
\[
\frac{|\rho|\sqrt {D}\log D}{\beta}\frac{y^{1-\beta}\log(y/z)}{z}
\]

Comparing the two estimates, we see they are approximately the same size when 
\[
\frac{z}{y^\beta}=\frac{\sqrt{D}y^{1-\beta}}{z},\quad\text{or}\quad z=D^{1/4}y^{1/2}.
\]

Combining the two sum estimates, and with 
\[
 \frac{1}{\beta}<\frac{6}{5},\quad\text{and}\quad 1<\frac{\log(y/\sqrt{D})\log D}{18},
\]
we have
\begin{multline*}
\frac{y^{1/2-\beta}|\rho|D^{1/4}}{\beta}+\frac{y^{1/2-\beta}|\rho|\log(y/\sqrt{D})D^{1/4}\log D}{2\beta}\\
<\frac{6}{5}\cdot\left(\frac{1}{18}+\frac12\right)y^{1/2-\beta}\log (y/\sqrt{D})|\rho|D^{1/4}\log D\\
= \frac{2}{3}\cdot y^{1/2-\beta}\log (y/\sqrt{D})|\rho|D^{1/4}\log D.
\end{multline*}
This concludes the Lemma.
\end{proof}
\subsection*{Lower Bounds, Again}
Applying Lemma \ref{L:Lemma2} with $y=U^{12}/n$, so $U^6<y<U^{12}$, we get
\begin{align*}
\Sigma_1<&8 U^{6-12\beta}\log U|\rho|D^{1/4}\log D \sum_{n\le U^{6}}\frac{\tau(n)}{\sqrt{n}}\\
=&8 U^{7-12\beta}\log U \sum_{n\le U^{6}}\frac{\tau(n)}{\sqrt{n}}.
\end{align*}
With an estimate  by the standard \lq method of the hyperbola\rq\, e.g.\ \cite[(2.9) p.37]{MV}, we get that
\[
 \sum_{n\le X}\frac{\tau(n)}{\sqrt{n}}=X^{1/2}\left(2\log X+4C-4\right)+O(1).
\]
Thus
\[
\Sigma_1< 96U^{10-12\beta}\log^2 U,
\]
and so, from $\beta>5/6$, is small.  In fact, from 
\[
\frac{1}{25}\frac{\zeta(4\beta)}{\zeta(2\beta)}-U^{6-12\beta}\le\Sigma_1+\Sigma_2,
\] 
\emph{Mathematica} tells us
$1/50<\Sigma_2$
when $\beta>7/8$ and $U>10^{16}$.  (We are assuming $D>10^9$, and Gourdan \cite{Gourdan} has verified the Riemann Hypothesis for the first $10^{13}$ zeros.  So our hypothetical $|\rho|>2.4\times 10^{12}$, so necessarily $U=|\rho|D^{1/4}\log D>10^{16}$.)

We now convert the lower bound for $\Sigma_2$ to a lower bound for $L(1,\chi)$.  Recall that
\[
\Sigma_2=\sum_{U^{6}< n\le U^{12}}\frac{1*\chi(n)}{n^\beta}\sum_{r\le U^{12}/n}\frac{\tau(r)}{r^\beta}.
\]
Writing $r^{-\beta}=r^{1-\beta}/r$ and using $r^{1-\beta}<U^{12(1-\beta)}n^{\beta-1}$ we see that
\[
\frac{1}{50}<\Sigma_2<U^{12(1-\beta)}\sum_{U^{6}< n\le U^{12}}\frac{1*\chi(n)}{n}\sum_{r\le U^{12}/n}\frac{\tau(r)}{r}.
\]
The \lq method of the hyperbola\rq\ argument shows in \cite[Ex. 11.2.1 (g)]{MV}\footnote{The implied constant in that exercise, combining six big Oh terms with implied constant equal $1$, can be taken to be $6$.} that
\begin{align*}
\sum_{U^{6}\le n\le U^{12}}\frac{1*\chi(n)}{n}=&\log(U^{6}) L(1,\chi)+O\left(D^{1/4} U^{-3}\log D\log(U^{6})\right)\\
=&\log(U^{6})L(1,\chi)+O\left(U^{-2}\log(U^{6})\right)\\
=&\log(U^{6})\left(L(1,\chi)+O\left(U^{-2}\right)\right).
\end{align*}
Meanwhile one more application of this same tool (along with Euler summation) gives that
\[
\sum_{r<X}\frac{\tau(r)}{r}=\frac{1}{2}\log^2 X+2C\log X+O(1).
\]
So
\[
\sum_{r\le U^{12}/n}\frac{\tau(r)}{r}\sim\frac12\log^2(U^{12}/n)<\frac12\log^2(U^{6}),
\]
as $U^{6}<n$.
Finally
\begin{align*}
\frac{1}{50}<\Sigma_2<&U^{12(1-\beta)}\log(U^{6})\left(L(1,\chi)+O\left(U^{-2}\right)\right)\cdot\frac12\log^2(U^{6})\\
=&108U^{12(1-\beta)}\log^3 U\left(L(1,\chi)+O\left(U^{-2}\right)\right).
\end{align*}

The implied constant is no worse than $6$, and  
\[
U^{-2}=\frac{1}{|\rho|^2\sqrt{D}\log^2D}<\frac{1}{\sqrt{D}},
\]
so the theorem follows.
\subsection*{The General Case}

We fix a prime $q|D$ and  consider
\[
\sum_{n=1}^\infty\frac{c_q(n)}{n^s}=\left(q^{1-s}-1\right)\zeta(s),
\]
where $c_q(n)$ is the Ramanujan sum
\[
c_q(n)=\sum_{k=1}^{q-1}\exp(2\pi ikn/q)=
\begin{cases}
-1&\text{ if } (n,q)=1\\
q-1&\text{ if } q|n.
\end{cases}
\]
(Observe that $c_2(n)=(-1)^n$.)\ \ Since $|\sum_{n<x} c_q(n) |<q$, the Dirichlet series converges conditionally for $\tre(s)>0$.  The Ramanujan sums are not multiplicative in $n$, but we have that $c_q(dm)=c_q(m)$ if $(d,q)=1$.  Instead of $\lodd$ we define a function $\lambda_q(n)=0$ if $q|n$.  The proof goes through as before.  We find that in Lemma \ref{L:Lemma1} we have that
\[
\sum_{n=1}^\infty\frac{c_q(n)\cdot 1*\lambda_q(n)}{n^\rho}=\left(1+q^{-\rho}\right)\left(1-q^{1-\rho}\right)\zeta(2\rho).
\]
so the trivial zeros along $\tre(s)=1$ when $\gamma=2\pi n/\log q$ still cause a problem.  In fact, the constant $1/25$ in Lemma \ref{L:Lemma1} which works for $q=2$ is a decreasing function of $q$ in the general case.

\end{document}